\documentclass[reqno]{amsart}

\usepackage{amsmath,amssymb,amsthm,amsfonts,mathrsfs}
\usepackage{fullpage}
\usepackage{xcolor}
\usepackage{graphicx}
\usepackage{listings}
\usepackage{bbm}

\definecolor{codegreen}{rgb}{0,0.6,0}
\definecolor{codegray}{rgb}{0.5,0.5,0.5}
\definecolor{codepurple}{rgb}{0.58,0,0.82}
\definecolor{backcolour}{rgb}{0.95,0.95,0.92}

\lstdefinestyle{mystyle}{
  backgroundcolor=\color{backcolour},   commentstyle=\color{codegreen},
  keywordstyle=\color{magenta},
  numberstyle=\tiny\color{codegray},
  stringstyle=\color{codepurple},
  basicstyle=\ttfamily\footnotesize,
  breakatwhitespace=false,         
  breaklines=true,                 
  captionpos=b,                    
  keepspaces=true,                 
  numbers=left,                    
  numbersep=5pt,                  
  showspaces=false,                
  showstringspaces=false,
  showtabs=false,                  
  tabsize=2
}

\usepackage{listings}
\usepackage[colorlinks=true,
linkcolor=blue,
anchorcolor=blue,
citecolor=red
]{hyperref}
\allowdisplaybreaks 
\newtheorem{theorem}{Theorem}[section]
\newtheorem{lemma}[theorem]{Lemma}

\newtheorem*{conjecture*}{Conjecture}
\theoremstyle{definition}

\theoremstyle{remark}

\newtheorem*{remark*}{remark}

\usepackage{colonequals}

\author{Runbo Li}
\address{International Curriculum Center, The High School Affiliated to Renmin University of China, Beijing, China}
\email{runbo.li.carey@gmail.com}

\makeatletter
\@namedef{subjclassname@2020}{\textup{}2020 Mathematics Subject Classification}
\makeatother

\title[]{On almost primes in Piatetski--Shapiro sequences}
\subjclass[2020]{11N35, 11N36} 
\keywords{prime, sieve methods, Piatetski--Shapiro sequences}

\begin{document}
	
\begin{abstract}
The author proves that for $0.9985 < \gamma < 1$, there exist infinitely many primes $p$ such that $[p^{1/\gamma}]$ has at most 5 prime factors counted with multiplicity. This gives an improvement upon the previous results of Banks--Guo--Shparlinski and Xue--Li--Zhang.
\end{abstract}

\maketitle

\tableofcontents

\section{Introduction}
Let $P_r$ denotes an integer with at most $r$ prime factors counted with multiplicity, $p$ denotes a prime and $\gamma \in (0, 1)$. The Piatetski--Shapiro sequences are sequences of the form
\begin{equation}
\mathscr{N}_{\gamma} = \left\{[n^{1/\gamma}]: n \in \mathbb{N}^{+} \right\}.
\end{equation}
In 1953, Piatetski--Shapiro \cite{PS1211} established that there are infinitely many primes in $\mathscr{N}_{\gamma}$ if $\gamma \in \left(\frac{11}{12}, 1\right)$. This range of $\gamma$ has been improved by many authors, and the best record now is due to Li , where he showed that the same result holds for any $\gamma \in \left(\frac{775}{919}, 1\right)$.

In 1987, Balog \cite{Balog89} considered the following subsequence of $\mathscr{N}_{\gamma}$:
\begin{equation}
\mathscr{P}_{\gamma} = \left\{[p^{1/\gamma}]: p \text{ prime} \right\}.
\end{equation}
He showed that for almost all $\gamma \in (0, 1)$, we have
\begin{equation}
\limsup_{x \to \infty} \frac{\sum_{\substack{q \leqslant x \\ q \in \mathscr{P}_{\gamma} \\ q \text{ prime} }} 1 }{x^{\gamma} \gamma^{-1} (\log x)^{-2}} \geqslant 1,
\end{equation}
but this result gives no information for any specific choice of $\gamma$.

In 2016, Banks, Guo and Shparlinski \cite{BanksGuoShparlinski} generalized and enhanced Balog's result. They showed explicitly that for every $\gamma \in (0,1) \backslash \left\{\frac{1}{z}: z \in \mathbb{Z}^{+} \right\}$, there exist infinitely many $q \in \mathscr{P}_{\gamma}$ such that $q = P_{R(\gamma)}$ for some finite $R(\gamma)$. As a part of their result, they showed that there are infinitely many $q \in \mathscr{P}_{\gamma}$ with $q = P_8$ for any $\gamma \in \left(0.9505, 1\right)$. In 2024, Xue, Li and Zhang \cite{XueLiZhang} improved this to $q = P_7$ for any $\gamma \in \left(0.989, 1\right)$. In this paper, we shall further improve this to $q = P_5$ when $\gamma$ is near 1.

\begin{theorem}\label{t1}
There are infinitely many $q \in \mathscr{P}_{\gamma}$ with $q = P_5$ for any $\gamma \in \left(0.9985, 1\right)$. Moreover, the following estimates
$$
\sum_{\substack{q \leqslant x \\ q \in \mathscr{P}_{\gamma} \\ \Omega(q) \leqslant 5 }} 1 \gg \frac{x^{\gamma}}{(\log x)^2}
$$
holds for all sufficiently large $x$ provided that $\gamma \in \left(0.9985, 1\right)$.
\end{theorem}

Throughout this paper, we always suppose that $x$ is a sufficiently large integer, $\varepsilon$ and $\eta$ are sufficiently small positive numbers and $0 < \gamma < 1$. The letter $p$, with or without subscript, is reserved for prime numbers. Let $\xi = \frac{140 \gamma - 99}{270}$, $u = \frac{1}{\xi} +\varepsilon$ and $\lambda = \frac{1}{9 - u - \varepsilon}$. We define the set $\mathcal{A}$ as
$$
\mathcal{A} = \left\{a: a \leqslant x, a \in \mathscr{P}_{\gamma} \right\}
$$
and we put
$$
\mathcal{A}_d=\{a: a d \in \mathcal{A}\}, \quad P(z)=\prod_{p<z} p, \quad S(\mathcal{A}, z)=\sum_{\substack{a \in \mathcal{A} \\ (a, P(z))=1}} 1.
$$

\section{Two combinatorial lemmas}
In this section we shall prove two combinatorial lemmas which will be used to handle some error terms after performing Chen's switching principle on some sieve functions.

\begin{lemma}\label{l21}
For any positive numbers $t_1, t_2, t_3, t_4, t_5, t_6, t_7$ with 
$$
\frac{1}{17.41} \leqslant t_1 < t_2 < t_3 < t_4 < t_5 < t_6 < t_7 \quad \text{and} \quad t_1 + t_2 + t_3 + t_4 + t_5 + t_6 + t_7 = 1,
$$
There exists some $I \subset \{1, 2, 3, 4, 5, 6, 7\}$ such that
$$
0.611797 \leqslant \sum_{i \in I} t_i \leqslant 0.787393.
$$
\end{lemma}
\begin{proof}
Consider the following cases:

1. $0.611797 \leqslant t_4 + t_5 + t_6 + t_7 \leqslant 0.787393$. Take $I = \{4, 5, 6, 7\}$.

2. $t_4 + t_5 + t_6 + t_7 > 0.787393$. Now we have 
$$
t_1 + t_2 + t_3 < 0.212607.
$$
Since
$$
\frac{3}{17.41} \leqslant 3 t_1 < t_1 + t_2 + t_3,
$$
we have
$$
0.787393 < t_4 + t_5 + t_6 + t_7 < 1 - \frac{3}{17.41} < 0.82769.
$$
Now 
$$
\frac{1}{17.41} < t_4 < \frac{1}{4}(t_4 + t_5 + t_6 + t_7) = 0.2069225,
$$
we have
$$
0.5804705 = 0.787393 - 0.2069225 < t_5 + t_6 + t_7 < 0.82769 - \frac{1}{17.41} < 0.770252.
$$
If $0.611797 \leqslant t_5 + t_6 + t_7 < 0.770252$, take $I = \{5, 6, 7\}$. Otherwise we have
$$
0.5804705 < t_5 + t_6 + t_7 < 0.611797.
$$
Since
$$
0.05743 < \frac{1}{17.41} \leqslant t_1 < \frac{1}{7}(t_1 + t_2 + t_3 + t_4 + t_5 + t_6 + t_7) = \frac{1}{7} < 0.143,
$$
we have
$$
0.6379005 = 0.05743 + 0.5804705 < t_1 + t_5 + t_6 + t_7 < 0.611797 + 0.143 = 0.754797.
$$
Take $I = \{1, 5, 6, 7\}$.

3. $t_4 + t_5 + t_6 + t_7 < 0.611797$. Now we have
$$
\frac{4}{17.41} < t_4 + t_5 + t_6 + t_7 < 0.611797
$$
and
$$
0.388203 = 1 - 0.611797 < t_1 + t_2 + t_3.
$$
Hence
$$
0.1294 < \frac{0.388203}{3} < \frac{1}{3}(t_1 + t_2 + t_3) < t_3 < t_4 < t_5.
$$
Note that
$$
t_1 + t_2 + t_3 < \frac{1}{2}(t_1 + t_2 + t_3 + t_4 + t_5 + t_6) < \frac{1}{2} \times \frac{6}{7} = \frac{3}{7} < 0.4286
$$
and
$$
t_4 + t_5 < \frac{1}{2}(t_4 + t_5 + t_6 + t_7) < \frac{0.611797}{2} < 0.3059,
$$
we have
$$
0.647 < 0.388203 + 2 \times 0.1294 < t_1 + t_2 + t_3 + t_4 + t_5
$$
and
$$
t_1 + t_2 + t_3 + t_4 + t_5 < 0.4286 + 0.3059 = 0.7345.
$$
Take $I = \{1, 2, 3, 4, 5\}$.

Combining all above cases, Lemma~\ref{l21} is proved.
\end{proof}

\begin{lemma}\label{l22}
For any positive numbers $t_1, t_2, t_3, t_4, t_5, t_6$ with 
$$
\frac{1}{17.41} \leqslant t_1 < t_2 < t_3 < t_4 < t_5 < t_6 \quad \text{and} \quad t_1 + t_2 + t_3 + t_4 + t_5 + t_6 = 1,
$$
There exists some $I \subset \{1, 2, 3, 4, 5, 6\}$ such that
$$
0.611797 \leqslant \sum_{i \in I} t_i \leqslant 0.787393.
$$
\end{lemma}
\begin{proof}
Consider the following cases:

1. $0.611797 \leqslant t_4 + t_5 + t_6 \leqslant 0.787393$. Take $I = \{4, 5, 6\}$.

2. $t_4 + t_5 + t_6 > 0.787393$. Now we have
$$
\frac{3}{17.41} < t_1 + t_2 + t_3 < 1 - 0.787393 = 0.212607
$$
and
$$
0.787393 < t_4 + t_5 + t_6 < 1 - \frac{3}{17.41} < 0.8277.
$$
We also have
$$
t_5 + t_6 < 1 - \frac{4}{17.41} < 0.7703.
$$
If $0.611797 \leqslant t_5 + t_6 < 0.7703$, take $I = \{5, 6\}$. Otherwise we have
$$
t_5 + t_6 < 0.611797.
$$
Since
$$
0.787393 < t_4 + t_5 + t_6,
$$
we have
$$
0.175596 < t_4 < t_5.
$$
Note that
$$
t_4 + t_5 < \frac{2}{3}(t_4 + t_5 + t_6) < \frac{2}{3}0.8277 = 0.5518,
$$
we have
$$
0.351192 = 2 \times 0.175596 < t_4 + t_5 < 0.5518.
$$
Since
$$
0.1723 < \frac{3}{17.41} < t_1 + t_2 + t_3 < 0.212607,
$$
we have
$$
0.523492 = 0.1723 + 0.351192 < t_1 + t_2 + t_3 + t_4 + t_5 < 0.212607 + 0.5518 = 0.764407.
$$
If $0.611797 \leqslant t_1 + t_2 + t_3 + t_4 + t_5 < 0.764407$, take $I = \{1, 2, 3, 4, 5\}$. Otherwise we have
$$
0.523492 < t_1 + t_2 + t_3 + t_4 + t_5 < 0.611797,
$$
which means that
$$
0.388203 = 1 - 0.611797 < t_6 < 1 - 0.523492 = 0.476508.
$$
If there exists $i \in \{1,2,3,4,5\}$ such that $0.212667 \leqslant t_i \leqslant 0.388203$, then we can take $I = \{1, \cdots, 6, \text{ except } i\}$. Now we consider the following 3 subcases:

2.1. At least two of $t_1, t_2, t_3, t_4, t_5$ are larger than $0.388203$. Now we have
$$
1.16 < 3 \times 0.388203 < t_1 + t_2 + t_3 + t_4 + t_5 + t_6 = 1,
$$
which is a contradiction.

2.2. One of $t_1, t_2, t_3, t_4, t_5$ is larger than $0.388203$. Now we must have
$$
\frac{1}{17.41} \leqslant t_1 < t_2 < t_3 < t_4 < 0.212667 < 0.388203 < t_5.
$$
Hence
$$
t_1 + t_2 + t_3 + t_4 < 1 - (t_5 + t_6) < 1 - 2 \times 0.388203 = 0.223594.
$$
If $0.212667 \leqslant t_1 + t_2 + t_3 + t_4 < 0.223594$, take $I = \{5, 6\}$. Otherwise we have
$$
t_1 + t_2 + t_3 + t_4 < 0.212667.
$$
But since
$$
0.2297 < \frac{4}{17.41} < t_1 + t_2 + t_3 + t_4,
$$
we get
$$
0.2297 < t_1 + t_2 + t_3 + t_4 < 0.212667,
$$
which is a contradiction.

2.3. None of $t_1, t_2, t_3, t_4, t_5$ is larger than $0.388203$. Now we must have
$$
\frac{1}{17.41} \leqslant t_1 < t_2 < t_3 < t_4 < t_5 < 0.212667
$$
and
$$
0.351192 < t_4 + t_5.
$$
Hence
$$
\frac{3}{17.41} < t_1 + t_2 + t_3 < 1 - (t_4 + t_5 + t_6) < 1 - 0.351192 - 0.388203 = 0.260605.
$$
If $0.212667 \leqslant t_1 + t_2 + t_3 < 0.260605$, take $I = \{4, 5, 6\}$. Otherwise we have
$$
\frac{3}{17.41} < t_1 + t_2 + t_3 < 0.212667.
$$
Now we have
$$
\frac{1}{17.41} \leqslant t_1 < \frac{1}{3}(t_1 + t_2 + t_3) < 0.0709.
$$
Since
$$
0.175596 < t_4 < 0.212607,
$$
we have
$$
0.233 < \frac{1}{17.41} + 0.175596 < t_1 + t_4 < 0.0709 + 0.212607 < 0.3.
$$
Take $I = \{2, 3, 5, 6\}$.

3. $t_4 + t_5 + t_6 < 0.611797$. Now we have
$$
0.5 = \frac{1}{2}(t_1 + t_2 + t_3 + t_4 + t_5 + t_6) < t_4 + t_5 + t_6 < 0.611797
$$
and
$$
\frac{1}{17.41} \leqslant t_1 < \frac{1}{6}(t_1 + t_2 + t_3 + t_4 + t_5 + t_6) = \frac{1}{6}.
$$
Hence
$$
0.5574 < \frac{1}{17.41} + 0.5 < t_1 + t_4 + t_5 + t_6 < \frac{1}{6} + 0.611797 < 0.7784.
$$
If $0.611797 \leqslant t_1 + t_4 + t_5 + t_6 < 0.7784$, take $I = \{1, 4, 5, 6\}$. Otherwise we have
$$
0.5574 < t_1 + t_4 + t_5 + t_6 < 0.611797.
$$
Since
$$
\frac{2}{17.41} \leqslant t_1 + t_2 < \frac{1}{3}(t_1 + t_2 + t_3 + t_4 + t_5 + t_6) = \frac{1}{3},
$$
we have
$$
0.6148 < \frac{2}{17.41} + 0.5 < t_1 + t_2 + t_4 + t_5 + t_6 < \frac{1}{3} + 0.611797 < 0.94514.
$$
If $0.6148 < t_1 + t_2 + t_4 + t_5 + t_6 \leqslant 0.787393$, take $I = \{1, 2, 4, 5, 6\}$. Otherwise we have
$$
0.787393 < t_1 + t_2 + t_4 + t_5 + t_6 < 0.94514.
$$
Hence
$$
\frac{1}{17.41} < t_3 < 1 - 0.787393 = 0.212607.
$$
Note that
$$
t_1 + t_4 + t_5 + t_6 < 0.611797,
$$
we have
$$
0.175596 = 0.787393 - 0.611797 < t_2.
$$
Hence
$$
0.87798 = 5 \times 0.175596 < 5 t_2 < t_2 + t_3 + t_4 + t_5 + t_6
$$
and
$$
\frac{1}{17.41} \leqslant t_1 < 1 - 0.87798 = 0.12202.
$$
Since
$$
0.175596 < t_2 < t_3 < 0.212607,
$$
we have
$$
0.22 < \frac{1}{17.41} + 0.175596 < t_1 + t_3 < 0.12202 + 0.212607 < 0.34.
$$
Take $I = \{2, 4, 5, 6\}$.

Combining all above cases, Lemma~\ref{l22} is proved.
\end{proof}

\section{Proof of Theorem 1.1}
Now we follow the discussion in \cite{XueLiZhang}. Consider the following weighted sum
\begin{equation}
W\left(\mathcal{A}, x^{\frac{1}{17.41}} \right) = \sum_{\substack{ a \in \mathcal{A} \\ \left(a, P\left(x^{\frac{1}{17.41}}\right) \right)=1 }} \left(1 - \lambda \sum_{\substack{ x^{\frac{1}{17.41}} \leqslant p < x^{\frac{1}{u}} \\ p \mid a }}\left( 1-\frac{u \log p}{\log x} \right) \right) = \sum_{\substack{ a \in \mathcal{A} \\ \left(a, P\left(x^{\frac{1}{17.41}}\right) \right)=1 }} \mathscr{W}_a.
\end{equation}
We have
\begin{align}
\nonumber W\left(\mathcal{A}, x^{\frac{1}{17.41}} \right) =&\ \sum_{\substack{ a \in \mathcal{A} \\ \left(a, P\left(x^{\frac{1}{17.41}}\right) \right)=1 \\ \Omega(a) \leqslant 5 }} \mathscr{W}_a + \sum_{\substack{ a \in \mathcal{A} \\ \left(a, P\left(x^{\frac{1}{17.41}}\right) \right)=1 \\ \Omega(a) = 6 \\ \mu(a) \neq 0 }} \mathscr{W}_a + \sum_{\substack{ a \in \mathcal{A} \\ \left(a, P\left(x^{\frac{1}{17.41}}\right) \right)=1 \\ \Omega(a) = 7 \\ \mu(a) \neq 0 }} \mathscr{W}_a \\
\nonumber & + \sum_{\substack{ a \in \mathcal{A} \\ \left(a, P\left(x^{\frac{1}{17.41}}\right) \right)=1 \\ \Omega(a) = 8 \\ \mu(a) \neq 0 }} \mathscr{W}_a + \sum_{\substack{ a \in \mathcal{A} \\ \left(a, P\left(x^{\frac{1}{17.41}}\right) \right)=1 \\ 6 \leqslant \Omega(a) \leqslant 8 \\ \mu(a) = 0 }} \mathscr{W}_a + \sum_{\substack{ a \in \mathcal{A} \\ \left(a, P\left(x^{\frac{1}{17.41}}\right) \right)=1 \\ \Omega(a) \geqslant 9 }} \mathscr{W}_a \\
=&\ W_1 + W_2 + W_3 + W_4 + W_5 + W_6.
\end{align}
Our aim is to show that $W_1 > 0$.

Trivially, we have $\mathscr{W}_a < 0$ for $\Omega(a) \geqslant 9$. We also have $W_5 \ll x^{1-\frac{1}{17.41}+\varepsilon}$. Thus,
\begin{align}
\nonumber W_1 =&\ W\left(\mathcal{A}, x^{\frac{1}{17.41}} \right) - W_2 - W_3 - W_4 - W_5 - W_6 \\
\nonumber =&\ W\left(\mathcal{A}, x^{\frac{1}{17.41}} \right) - W_2 - W_3 - W_4 - W_6 + O\left(x^{1-10^{-10}}\right) \\
>&\ W\left(\mathcal{A}, x^{\frac{1}{17.41}} \right) - W_2 - W_3 - W_4 + O\left(x^{1-10^{-10}}\right).
\end{align}

By the same arguments as in \cite{XueLiZhang}, we have
\begin{align}
\nonumber W\left(\mathcal{A}, x^{\frac{1}{17.41}} \right) =&\ S\left(\mathcal{A}, x^{\frac{1}{17.41}} \right) - \lambda \sum_{ x^{\frac{1}{17.41}} \leqslant p < x^{\frac{1}{u}}}\left( 1-\frac{u \log p}{\log x} \right) S\left(\mathcal{A}_p, x^{\frac{1}{17.41}} \right) \\
\geqslant&\ (1+o(1)) \frac{2 \mathcal{C}(\omega) \pi\left(x^{\gamma}\right)}{17.41^2 \log x} \left(\frac{\log (17.41\xi -1)}{\xi} - \lambda \int_{u}^{17.41}\frac{t-u}{t(t \xi -1)} d t \right)
\end{align}
and
\begin{align}
\nonumber W_4 =&\ \sum_{\substack{ a \in \mathcal{A} \\ \left(a, P\left(x^{\frac{1}{17.41}}\right) \right)=1 \\ \Omega(a) = 8 \\ \mu(a) \neq 0 }} \left(1 - \lambda \sum_{\substack{ x^{\frac{1}{17.41}} \leqslant p < x^{\frac{1}{u}} \\ p \mid a }}\left( 1-\frac{u \log p}{\log x} \right) \right) \\
\nonumber \leqslant&\ \lambda \sum_{\substack{ a \in \mathcal{A} \\ \left(a, P\left(x^{\frac{1}{17.41}}\right) \right)=1 \\ \Omega(a) = 8 \\ \mu(a) \neq 0 }} 1 \\
\nonumber \leqslant&\ (1+o(1)) \frac{2 \mathcal{C}(\omega) \pi\left(x^{\gamma}\right)}{17.41^2 \log x}\left(\frac{\lambda \gamma}{\xi} \int_{\frac{1}{17.41}}^{\frac{1}{8}} \int_{t_1}^{\frac{1}{7}(1-t_1)} \int_{t_2}^{\frac{1}{6}(1-t_1-t_2)} \int_{t_3}^{\frac{1}{5}(1-t_1-t_2-t_3)} \int_{t_4}^{\frac{1}{4}(1-t_1-t_2-t_3-t_4)} \right. \\
\nonumber & \qquad \qquad \qquad \qquad \qquad \qquad \int_{t_5}^{\frac{1}{3}(1-t_1-t_2-t_3-t_4-t_5)} \int_{t_6}^{\frac{1}{2}(1-t_1-t_2-t_3-t_4-t_5-t_6)} \\
\nonumber & \left. \qquad \qquad \qquad \qquad \qquad \qquad \frac{1}{t_1 t_2 t_3 t_4 t_5 t_6 t_7 (1-t_1-t_2-t_3-t_4-t_5-t_6-t_7)} d t_7 d t_6 d t_5 d t_4 d t_3 d t_2 d t_1 \right) \\
\leqslant&\ (1+o(1)) \frac{2 \mathcal{C}(\omega) \pi\left(x^{\gamma}\right)}{17.41^2 \log x}\left(\frac{\lambda \gamma}{\xi} 0.00259 \right),
\end{align}
where $\mathcal{C}(\omega)$ is defined in [\cite{XueLiZhang}, (2.6)].

Note that we have
$$
5 - 5 \gamma + 4 \xi < 0.611797
$$
and
$$
\frac{1}{4}(\gamma + \xi + 2) > 0.787393
$$
when $0.9985 < \gamma < 1$. Then by Chen's switching principle, Iwaniec's linear sieve, Lemmas~\ref{l21}--\ref{l22} and similar arguments as in \cite{XueLiZhang}, we can obtain
\begin{align}
\nonumber W_3 =&\ \sum_{\substack{ a \in \mathcal{A} \\ \left(a, P\left(x^{\frac{1}{17.41}}\right) \right)=1 \\ \Omega(a) = 7 \\ \mu(a) \neq 0 }} \left(1 - \lambda \sum_{\substack{ x^{\frac{1}{17.41}} \leqslant p < x^{\frac{1}{u}} \\ p \mid a }}\left( 1-\frac{u \log p}{\log x} \right) \right) \\
\nonumber \leqslant&\ \lambda \sum_{\substack{ a \in \mathcal{A} \\ \left(a, P\left(x^{\frac{1}{17.41}}\right) \right)=1 \\ \Omega(a) = 7 \\ \mu(a) \neq 0 }} 1 \\
\nonumber \leqslant&\ (1+o(1)) \frac{2 \mathcal{C}(\omega) \pi\left(x^{\gamma}\right)}{17.41^2 \log x}\left(\frac{\lambda \gamma}{\xi} \int_{\frac{1}{17.41}}^{\frac{1}{7}} \int_{t_1}^{\frac{1}{6}(1-t_1)} \int_{t_2}^{\frac{1}{5}(1-t_1-t_2)} \int_{t_3}^{\frac{1}{4}(1-t_1-t_2-t_3)} \right. \\
\nonumber & \qquad \qquad \qquad \qquad \qquad \qquad \int_{t_4}^{\frac{1}{3}(1-t_1-t_2-t_3-t_4)} \int_{t_5}^{\frac{1}{2}(1-t_1-t_2-t_3-t_4-t_5)} \\
\nonumber & \left. \qquad \qquad \qquad \qquad \qquad \qquad \frac{1}{t_1 t_2 t_3 t_4 t_5 t_6 (1-t_1-t_2-t_3-t_4-t_5-t_6)} d t_6 d t_5 d t_4 d t_3 d t_2 d t_1 \right) \\
\leqslant&\ (1+o(1)) \frac{2 \mathcal{C}(\omega) \pi\left(x^{\gamma}\right)}{17.41^2 \log x}\left(\frac{\lambda \gamma}{\xi} 0.02571 \right)
\end{align}
and
\begin{align}
\nonumber W_2 =&\ \sum_{\substack{ a \in \mathcal{A} \\ \left(a, P\left(x^{\frac{1}{17.41}}\right) \right)=1 \\ \Omega(a) = 6 \\ \mu(a) \neq 0 }} \left(1 - \lambda \sum_{\substack{ x^{\frac{1}{17.41}} \leqslant p < x^{\frac{1}{u}} \\ p \mid a }}\left( 1-\frac{u \log p}{\log x} \right) \right) \\
\nonumber \leqslant&\ \lambda \sum_{\substack{ a \in \mathcal{A} \\ \left(a, P\left(x^{\frac{1}{17.41}}\right) \right)=1 \\ \Omega(a) = 6 \\ \mu(a) \neq 0 }} 1 \\
\nonumber \leqslant&\ (1+o(1)) \frac{2 \mathcal{C}(\omega) \pi\left(x^{\gamma}\right)}{17.41^2 \log x}\left(\frac{\lambda \gamma}{\xi} \int_{\frac{1}{17.41}}^{\frac{1}{6}} \int_{t_1}^{\frac{1}{5}(1-t_1)} \int_{t_2}^{\frac{1}{4}(1-t_1-t_2)} \int_{t_3}^{\frac{1}{3}(1-t_1-t_2-t_3)} \int_{t_4}^{\frac{1}{2}(1-t_1-t_2-t_3-t_4)} \right. \\
\nonumber & \left. \qquad \qquad \qquad \qquad \qquad \qquad \frac{1}{t_1 t_2 t_3 t_4 t_5 (1-t_1-t_2-t_3-t_4-t_5)} d t_5 d t_4 d t_3 d t_2 d t_1 \right) \\
\leqslant&\ (1+o(1)) \frac{2 \mathcal{C}(\omega) \pi\left(x^{\gamma}\right)}{17.41^2 \log x}\left(\frac{\lambda \gamma}{\xi} 0.16688 \right).
\end{align}

Finally, combining (6)--(10) we get that
\begin{align}
\nonumber W_1 >&\ W\left(\mathcal{A}, x^{\frac{1}{17.41}} \right) - W_2 - W_3 - W_4 + O\left(x^{1-10^{-10}}\right) \\
\nonumber \geqslant&\ (1+o(1)) \frac{2 \mathcal{C}(\omega) \pi\left(x^{\gamma}\right)}{17.41^2 \log x} \left(\frac{\log (17.41\xi -1)}{\xi} - \lambda \int_{u}^{17.41}\frac{t-u}{t(t \xi -1)} d t \right. \\
\nonumber & \left. \qquad \qquad \qquad \qquad \qquad \qquad - \frac{\lambda \gamma}{\xi} (0.00259+0.02571+0.16688) \right) + O\left(x^{1-10^{-10}}\right).
\end{align}
For $0.9985 < \gamma < 1$, we know that
\begin{equation}
\frac{\log (17.41\xi -1)}{\xi} - \lambda \int_{u}^{17.41}\frac{t-u}{t(t \xi -1)} d t - \frac{\lambda \gamma}{\xi} (0.00259+0.02571+0.16688) > 0.004
\end{equation}
and the proof of Theorem~\ref{t1} is completed.

We remark that the same method fails to prove $[p^{1/\gamma}] = P_4$ because of the following two restrictions:

1. We cannot use the estimation of exponential sums [\cite{XueLiZhang}, Lemma 5.1] to handle the error term occurred since we can not obtain a result of Lemma~\ref{l22}--type with 5 variables. An obvious counterexample is that each one of $p_1, p_2, p_3, p_4, p_5$ has size around $x^{0.2}$.

2. Even if we can enlarge the corresponding Type--II range in  so that a result of Lemma~\ref{l22}--type with 5 variables can be obtained, we cannot get a positive lower bound for
$$
\sum_{\substack{ a \in \mathcal{A} \\ \left(a, P\left(x^{\frac{1}{17.41}}\right) \right)=1 \\ \Omega(a) \leqslant 4 }} \mathscr{W}_a
$$
using Richert's logarithmic sieve weight.

It seems that the limit of our method is to prove $[p^{1/\gamma}] = P_5$.

\bibliographystyle{plain}
\bibliography{bib}

\end{document}